\documentclass[a4paper,UKenglish,cleveref, autoref, thm-restate]{lipics-v2021}

\hideLIPIcs

\title{A Simplified Proof for the Edge Density of 4-Planar Graphs}

\titlerunning{Edge Density of 4-Planar Graphs}

\author{Aaron B{\"u}ngener}{Universit{\"a}t T{\"u}bingen}{aaron.buengener@student.uni-tuebingen.de}{}{}

\authorrunning{A. B\"ungener}

\Copyright{Aaron B\"ungener}



\nolinenumbers

\DeclareMathOperator{\ch}{ch}

\newcommand{\si}{\textsf{Step~1}}
\newcommand{\sii}{\textsf{Step~2}}
\newcommand{\siii}{\textsf{Step~3}}
\newcommand{\siv}{\textsf{Step~4}}
\newcommand{\sv}{\textsf{Step~5}}

\begin{document}

\maketitle

\begin{abstract}
A graph on $n \ge 3$ vertices drawn in the plane such that each edge is crossed at most four times has at most $6(n-2)$ edges -- this result, proven by Ackerman, is outstanding in the literature of beyond-planar graphs with regard to its tightness and the structural complexity of the graph class. We provide a much shorter proof while at the same time relaxing the conditions on the graph and its embedding, i.e., allowing multi-edges and non-simple drawings.
\end{abstract}

\section{Introduction}\label{sec:introduction}
A drawing $\Gamma$ of a graph $G = (V,E)$ is an embedding in the plane such that each vertex $v \in V$ is mapped to a distinct point and each edge $e \in E$ to a Jordan arc connecting the corresponding endpoints. We assume that (i) no vertex is an interior point of an edge, (ii) two edges have only finitely many intersection points, which are either proper crossings or a common endpoint, and (iii) at most two edges cross at a single point. We allow parallel edges and loops in $G$, if they are non-homotopic in $\Gamma$, i.e., they form a closed curve with at least one vertex inside and one outside. We denote by $P(\Gamma)$ the planarization of $\Gamma$, i.e., the vertices and crossing points of $\Gamma$ are the vertices of $P(\Gamma)$, while the edges of $P(\Gamma)$ are the crossing-free segments in $\Gamma$, which are bounded by vertices and crossing points.

A given graph $G$ is $k$-planar, if it has a drawing $\Gamma$, where each edge is involved in at most $k$ crossings. Such graphs find direct application in the well-known Crossing Lemma \cite{DBLP:books/daglib/0019107, bungener2024improving, DBLP:journals/combinatorica/PachT97}. For $k \le 3$, tight (up to a constant) upper bounds on the edge density, i.e., the maximum number of edges a graph $G$ on $n$ vertices can have, were shown ($k=1$: $4(n-2)$ \cite{von1983bemerkungen}; $k=2$: $5(n-2)$ \cite{DBLP:conf/compgeom/Bekos0R17, DBLP:journals/combinatorica/PachT97}; $k=3$: $5.5(n-2)$ \cite{DBLP:conf/compgeom/Bekos0R17, bungener2024improving, DBLP:journals/dcg/PachRTT06}).
For $k=4$, an explicit proof of the tight upper bound $6(n-2)$ has been given only for the case where $G$ is simple \cite{DBLP:journals/comgeo/Ackerman19}. Unfortunately, the proof relies on a massive discharging argument with lengthy case analysis.
Here, we considerably simplify the proof and extend it to the non-simple case.

Recently, the known upper bound for the edge density of simple 5-planar graphs has been drastically improved from $\approx 8.3n$ to $7n$ \cite{bungener2025first}. A key ingredient for the simplification of the case analysis was a separate treatment of empty triangular faces described by three pairwise crossing edges in the drawing.
We apply this idea to the 4-planar setting. For 4-planar graphs, it is known that the upper bound of $6(n-2)$ can be achieved by a planar hexagonalization where each hexagon $h$ is enhanced by nine further crossing edges inside $h$, see \cref{fig:hexagon-def}. 
We specify this pattern $H^*$ in \cref{sec:def} and will show that we can assume that an empty triangular face is always part of an $H^*$ (\cref{pro:0-triangle}).

Another source for simplification compared to \cite{DBLP:journals/comgeo/Ackerman19} is the fact that just a few types of faces can exist in $P(\Gamma)$ (\cref{pro:small-faces,pro:fill,pro:cr-minimal,pro:cr-minimal-2}) -- which only holds in our more general setting. This leads to a short proof of the following theorem.

\begin{theorem}\label{th:4-planar}
    Any connected 4-planar graph $G$ on $n \ge 3$ vertices has at most $6(n-2)$ edges.
\end{theorem}

\section{Definitions}\label{sec:def}

Let $\Gamma$ be a 4-planar drawing of a graph $G=(V,E)$ and let $P(\Gamma)$ be the corresponding planarization.
For a face $f \in P(\Gamma)$, we define $\vert f \vert$ as the number of edges and $v(f)$ as the number of vertices of $G$ counted while traversing the boundary of $f$ (so double-counting is possible). We say that $f$ is a $v(f)$-face or, more specifically, a
$v(f)$-$\vert f \vert$-gon and use the terms \emph{$v(f)$-triangle, $v(f)$-quadrilateral} and \emph{$v(f)$-pentagon} for the cases of $\vert f \vert =3,4,5$ respectively.

\begin{figure}[t]
    \hfil
    \begin{subfigure}[b]{0.2\linewidth}
    \center
        \includegraphics[page=1, width = \linewidth]{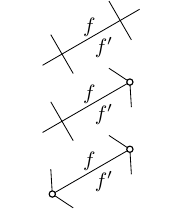}
        \subcaption{}
        \label{fig:definition-neighbors-a}
    \end{subfigure}  
    \hfil
    \begin{subfigure}[b]{0.2\linewidth}
    \center
        \includegraphics[page=2, width = \linewidth]{neighbors.pdf}
        \subcaption{}
        \label{fig:definition-neighbors-b}
    \end{subfigure} 
    \hfil
    \begin{subfigure}[b]{0.2\linewidth}
    \center
        \includegraphics[page=3, width = \linewidth]{neighbors.pdf}
        \subcaption{}
        \label{fig:definition-neighbors-c}
    \end{subfigure} 
    \hfil
        \caption{(Taken from \cite{bungener2024improving}.) Illustrations of the defined neighborhood relations. (a) From top to bottom: The faces $f$ and $f'$ are 0-neighbors, 1-neighbors, 2-neighbors, respectively. (b) The 0-pentagon $f_2$ is the wedge-neighbor of the 1-triangle $f_0$. (c) The faces $f$ and $f'$ are vertex-neighbors.}
        \label{fig:definition-neighbors}
\end{figure}

Several neighbor relationships in $P(\Gamma)$ will be useful later: We say that two faces $f,f'$ are \emph{$i$-neighbors} for $i \in \{0,1,2\}$ if they share in $P(\Gamma)$ an edge and $i$ vertices of $G$, see \cref{fig:definition-neighbors-a}.
Let $f_0 \in P(\Gamma)$ be a 1-triangle and $f_1$ its 0-neighbor. For $i\ge 1$, if $f_i$ is a 0-quadrilateral, then let $f_{i+1}$ be the 0-neighbor of $f$ at the edge opposite to $f_{i-1}$. For the largest such $i$ (which is $\le 4$ by 4-planarity), we call $f_i$ the \emph{wedge-neighbor} of $f_0$, see \cref{fig:definition-neighbors-b}. The set of all $f_i$ is the \emph{wedge} of $f_0$. Finally, we define two faces $f$ and $f'$ to be \emph{vertex-neighbors}, if they share a crossing $x$ but not an edge of $P(\Gamma)$ incident to $x$, see \cref{fig:definition-neighbors-c}.

Finally, we define $H^*$ to be 
a connected region of 19 faces in $\Gamma$ consisting of a central 0-triangle $t$, where all 0-neighbors of $t$ are 0-pentagons and all vertex-neighbors of $t$ are 0-quadrilaterals. Additionally, the two remaining 0-neighbors of the six 0-quadrilaterals and 0-pentagons are 1-triangles,
see \cref{fig:hexagon-def}. This definition is inspired by dense 2-planar and 3-planar configurations studied in \cite{bungener2024improving}.

\begin{figure}[t]
    \hfil
    \center
    \includegraphics[page=1, width = 0.3\linewidth]{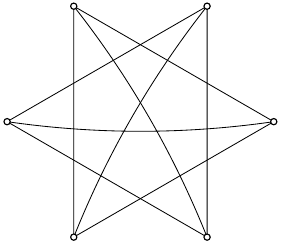}
    \hfil
    \caption{An $H^*$ configuration. The six vertices are not necessarily distinct; see, e.g., \cref{fig:hexagon-b}.}
    \label{fig:hexagon-def}
\end{figure}

\section{Proof of \cref{th:4-planar}}\label{sec:proof}

\begin{proof}

Let $\Gamma$ be a 4-planar drawing of an edge-maximal $n$-vertex graph $G$ that maximizes the number of $H^*$ configurations and, subject to this condition, minimizes the number of crossings.

We start by specifying which types of faces can occur in $P(\Gamma)$.
For that, we prove the theorem by induction on the number of vertices (this is needed for \cref{pro:2-connected}), 
so observe that for the base case $n = 3$ at most six non-homotopic edges are possible (three loops and three non-loop-edges). The next proposition is inspired by \cite[Proposition 2.1]{DBLP:journals/comgeo/Ackerman19}, but we have to adapt it to the non-simple case.

\begin{proposition}\label{pro:2-connected}
    Assume that $P(\Gamma)$ has a vertex $x$ that is a crossing in $\Gamma$ so that $P(\Gamma) \setminus {x}$ is not connected. Then $G$ has at most $6(n-2)$ edges.
\end{proposition}

\begin{proof}
    Let $\tilde{\Gamma}$ be the drawing of the graph $\tilde{G}$ obtained by replacing the crossing $x$ by a vertex. This adds one vertex and two edges. Let $\tilde{\Gamma}_1, ..., \tilde{\Gamma}_k$ be the connected components of $\tilde{\Gamma} \setminus \{x\}$ with corresponding graphs $G_1, ..., G_k$. Let $\Gamma'$ be the drawing  induced by the vertices of $G_1$ and $x$, and let $\Gamma''$ be the drawing induced by the vertices of $G_2 \cup ... \cup G_k$ and $x$. Let $G',G''$ be the graphs corresponding to $\Gamma', \Gamma''$. Note that both have vertex $x$.

    Under this construction, no multi-edges of $G'$ or $G''$ are homotopic in $\Gamma'$ or $\Gamma''$, because the vertex $x$ separates the same Jordan arcs in $\Gamma'$ that were separated in $\Gamma$ by vertices of $G''$ and vice versa.

    Let $n',n''$ be the number of vertices of $G',G''$. By construction, we have $n'+n'' = n+2\ge 5$ and w.l.o.g. $n' \ge 3$ and $n''\ge 2$. If $n''=2$, we observe that $n=n'$ while $G'$ has one edge more than $G$. This contradicts that $G$ is maximally dense. If $n'' \ge 3$, we can apply induction and conclude that $G$ has at most $6n'-12 + 6n''-12-2 = 6(n+2)-26 \le 6n-12$ edges.
\end{proof}

Thus, if the boundary of a face is not a simple cycle, then it is a vertex of $G$ that appears multiple times.

\begin{proposition}\label{pro:small-faces}
    For each face $f \in P(\Gamma)$, we have $\vert f \vert \ge 3$.
\end{proposition}
\begin{proof}
A 0-1-gon implies a self-crossing edge, which is forbidden as the edges are Jordan arcs. A 1-1-gon is created by an homotopic loop and a 2-2-gon by homotopic multi-edges, which both are not allowed. In the case that $f$ is a 0-2-gon or a 1-2-gon, we have a contradiction to the crossing-minimality of $\Gamma$ by swapping the two edge segments forming $f$.
\end{proof}

\begin{proposition}\label{pro:fill}
    A face $f \in P(\Gamma)$ with $ v(f) > 2$ is a 3-triangle.
\end{proposition}
\begin{proof}
    Let $v_1, v_2, v_3$ be the (not necessarily distinct) vertices of $f$. If $\vert f \vert \ge 4$, then not all edges $v_iv_j$ with $i,j \in \{1,2,3\}$ are present on the boundary of $f$. Hence, we may insert such an edge (which is not homotopic to any other edge), contradicting that $G$ is edge-maximal.
\end{proof}

\begin{figure}[t]
    \hfil
    \begin{subfigure}[b]{0.3715\linewidth}
    \center
        \includegraphics[page=1, width = \linewidth]{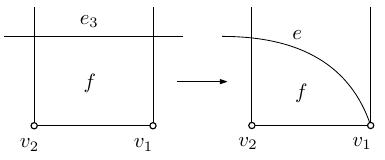}
        \subcaption{}
        \label{fig:face-reduction-a}
    \end{subfigure}
    \hfil
    \begin{subfigure}[b]{0.21\linewidth}
    \center
        \includegraphics[page=3, width = \linewidth]{rerouting.pdf}
        \subcaption{}
        \label{fig:face-reduction-b}
    \end{subfigure} \\
    \hfil
    \begin{subfigure}[b]{0.3715\linewidth}
    \center
        \includegraphics[page=2, width = \linewidth]{rerouting.pdf}
        \subcaption{}
        \label{fig:face-reduction-c}
    \end{subfigure} 
    \hfil
    \begin{subfigure}[b]{0.21\linewidth}
    \center
        \includegraphics[page=4, width = \linewidth]{rerouting.pdf}
        \subcaption{}
        \label{fig:face-reduction-d}
    \end{subfigure} 
    \hfil
    \begin{subfigure}[b]{0.21\linewidth}
    \center
        \includegraphics[page=5, width = \linewidth]{rerouting.pdf}
        \subcaption{}
        \label{fig:face-reduction-e}
    \end{subfigure} 
    \hfil
        \caption{Illustrations for (a)--(b) \cref{pro:cr-minimal} and (c)--(e) \cref{pro:cr-minimal-2}.}
        \label{fig:face-reduction}
\end{figure}

\begin{proposition}\label{pro:cr-minimal}
    A face $f \in P(\Gamma)$ with $ v(f) = 2$ is a 2-triangle.
\end{proposition}
\begin{proof}
    Let $f$ be a face with vertices $v_1,v_2$ and assume that $\vert f \vert \ge 4$. As in the proof above, we can add the planar edge $v_1v_2$ if not present on the boundary of $f$ and therefore assume its existence. Let $e_i$, $i \in \{1,2,3, ... \vert f \vert\},$ denote the edges on the boundary of $f$ starting with $e_1 = v_1v_2$. Because of $\vert f \vert \ge 4$, the edge segment of $e_3$ at $f$ is not incident to $v_1$.
    
    Let $x$ resp., $y$ be the crossings of $e_2$ and $e_3$  resp., $e_3$ and $e_4$ and let $u,v$ be the end-vertices of $e_3$, so that $x$ is visited before $y$ when walking along $e_3$ starting at $u$. 
    We replace $e_3$ by an edge $e$ that consists of the edge segments of $e_3$ between $u$ and $x$ and the segment $(x,v_1)$, see \cref{fig:face-reduction-a}. While this may introduce a multi-edge or a loop, the resulting edge $e$ can not be homotopic to an existing edge. To see this, assume that $e$ and an edge $e'$ are homotopic in the new drawing. Then, in $\Gamma$, $e$ and $e_3$ must cross, and swapping of their edge segments between their common vertex and the crossing would yield a drawing with fewer crossings, see \cref{fig:face-reduction-b}.
    
    The resulting drawing has fewer crossings, contradicting the assumption of crossing-minimality (while the numbers of edges and $H^*$ configurations do not change).
\end{proof}

Combining this with \cref{pro:2-connected}, we know that the boundary of all faces except 2-triangles and 3-triangles are simple cycles. The non-simple 2-triangles have exactly one vertex on its boundary, which is connected by a loop. There is no need to handle this type separately, as all following arguments work the same for the two types of 2-triangles. Non-simple 3-triangles can be treated the same way.

\begin{proposition}\label{pro:cr-minimal-2}
    A face $f \in P(\Gamma)$ with $ v(f) = 1$ is a 1-triangle or 1-quadrilateral.
\end{proposition}
\begin{proof}
    Let $f$ be a face with vertex $v$ and assume that $\vert f \vert \ge 5$. Let $e_i$, $i \in \{1,2,3,...,\vert f \vert\}$ denote the edges on the boundary of $f$ in clockwise order, starting with $e_1$ incident to $v$.
    Then $e_2$ is not incident to $v$ at $f$.

    Let $x$ resp., $y$ be the crossings of $e_1$ and $e_2$ resp., $e_2$ and $e_3$ and let $u,v$ be the end-vertices of $e_2$, so that $x$ is visited before $y$ when walking along $e_2$ starting at $u$.   
    We replace $e_2$ by an edge $e$ that consists of the edge segments of $e_2$ between $u$ and $x$ and the segment $(x,v)$, see \cref{fig:face-reduction-c}. While this may introduce a multi-edge or a loop, the resulting edge $e$ can not be homotopic to an existing edge. To see this, assume that $e$ and an edge $e'$ are homotopic in the new drawing. Then, in $\Gamma$, either $e_2$ and $e'$ cross in a way that is not crossing-minimal, see \cref{fig:face-reduction-d}, or $e'$ has two crossings with $e_4$ and swapping the edge segments between the crossings would yield a drawing with fewer crossings, see \cref{fig:face-reduction-e}.
    
    The resulting drawing has fewer crossings and the same number of edges and $H^*$ configurations, thus we found a contradiction to crossing-minimality.
\end{proof}

The crucial observation that leads to the simplification of the proof compared to \cite{DBLP:journals/comgeo/Ackerman19} is the next proposition.

\begin{proposition}\label{pro:0-triangle}
    Every 0-triangle $t \in P(\Gamma)$ is part of an $H^*$ configuration.
\end{proposition}

\begin{proof}
    Assume that $t$ is not part of an $H^*$. Let $e_1, e_2, e_3 \in E$ be the three edges that form $t$. These edges are distinct, as self-crossings are forbidden. Let $v_i, i\in\{1,...,6\}$ be the endpoints of $e_1,e_2,e_3$, which are not necessarily distinct but nevertheless define a (non-simple), empty hexagon $h$ as following:
    Consider all $v_i$ as different vertices, even if one $v_i$ may appear multiple times as an endpoint of edges $e_1, e_2, e_3$. Then we choose an arbitrary tree $T$ on these six vertices using only edge segments of edges $e_i$. Such a tree exists as all edge segments of the edges $e_i$ are connected in $P(\Gamma)$. Now define $h$ as the region that closely surrounds $T$, see \cref{fig:hexagon-a}.     
    
    Replace the edges that lie (partially) in $h$ by an $H^*$ configuration, see \cref{fig:hexagon-b}. All the replaced edges must cross one of $e_1,e_2,e_3$ and by 4-planarity, there are at most nine of them. Thus, this procedure does not decrease the number of edges, while it increases the number of $H^*$ configurations.
\begin{figure}[t]
    \hfil
    \begin{subfigure}[b]{0.3\linewidth}
    \center
        \includegraphics[page=2, width = \linewidth]{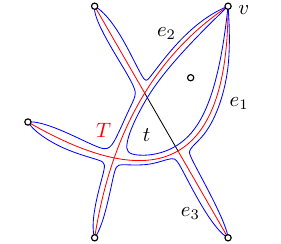}
        \subcaption{}
        \label{fig:hexagon-a}
    \end{subfigure}  
    \hfil
    \begin{subfigure}[b]{0.3\linewidth}
    \center
        \includegraphics[page=3, width = \linewidth]{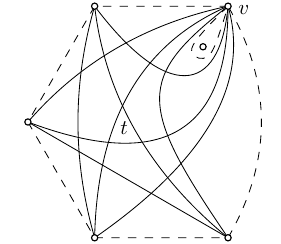}
        \subcaption{}
        \label{fig:hexagon-b}
    \end{subfigure} 
    \hfil
        \caption{(a) The three edges forming $t$ define a (non-simple) hexagon $h$ (blue) by following closely the edge segments of a tree $T$ (red) of $P(\Gamma)$. If $e_1,e_2$ have a common endpoint $v$, then they must enclose a further vertex, as otherwise the number of crossings could be reduced. (b) The same vertices and triangle $t$ after inserting the $H^*$ configuration. The dashed edges must exist by \cref{pro:fill}.}
        \label{fig:hexagon}
\end{figure}
\end{proof}

For the main part of the proof, we use the discharging method \cite{ackerman2006maximum,DBLP:journals/comgeo/Ackerman19}. First, we assign $\ch(f) = \vert f \vert +  v(f) -4$ charge to the faces of $P(\Gamma)$, thereby a total charge of $4n-8$ is distributed \cite{DBLP:journals/jct/AckermanT07}. Then, we redistribute the charge so that each edge of $G$ receives at least $\frac23$ and no face in $P(\Gamma)$ has a negative charge. From this we will conclude that the number of edges in $G$ is at most $\frac{4}{2/ 3}(n-2) = 6(n-2)$. 
More precisely, we will show equivalently that the final charge $\ch'(f)$ of every face is at least $\frac{1}{3}v(f)$ and do not transfer the charge explicitly to the edges. Initially, all faces except 0-triangles and 1-triangles have enough charge.

\medskip

\noindent \textbf{Discharging:}
We will distribute the charge respecting the property that charge is never moved through planar edges of $\Gamma$. Therefore, we can study the charge of the faces that lie inside and outside a cycle of planar edges separately. For $H^*$ configurations (which are surrounded by a planar cycle by \cref{pro:fill}), we simply observe that the total charge given and needed both is eight and thus a valid reassigning is possible, see \cref{fig:hexagon-charging}.
\begin{figure}[t]
    \hfil
    \begin{subfigure}[b]{0.30\linewidth}
    \center
        \includegraphics[page=11, width = \linewidth]{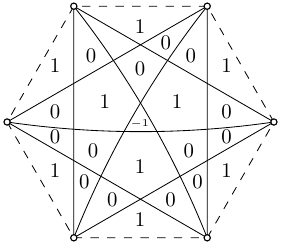}
        \subcaption{}
        \label{}
    \end{subfigure}  
    \hfil
    \begin{subfigure}[b]{0.30\linewidth}
    \center
        \includegraphics[page=12, width = \linewidth]{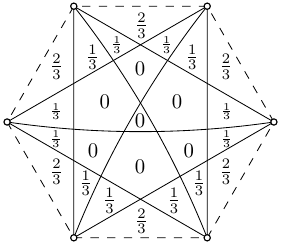}
        \subcaption{}
    \end{subfigure} 
    \hfil
    \caption{(a) Initial and (b) redistributed charges of an $H^*$ configuration.}
    \label{fig:hexagon-charging}
\end{figure}

For all faces that are not part of an $H^*$ configuration (which excludes 0-triangles by \cref{pro:0-triangle}), we state explicit discharging rules.
The redistribution of charge takes place in five steps and we denote by $\ch_i(f)$ the charge of a face $f$ after the $i$-th step, where $\ch'(f) := \ch_5(f)$ is the final charge.
We say that a face $f$ is \textit{satisfied} after step $i$, if $\ch_i(f) \ge \frac13 v(f)$ and define the \textit{excess} to be $\ch_i(f) - \frac13 v(f)$.
The initial charge $\ch(f)$ is redistributed according to the following steps:
\begin{itemize}
    \item \si: Each 1-triangle $f$ receives $\frac{1}{9}$ charge from its 1-neighbors that are 2-triangles or 1-quadrilaterals. If the two 1-neighbors of $f$ are one 2-triangle and one 1-quadrilateral, $f$ only receives charge from the 2-triangle.
    \item \sii: Each 1-triangle $f$ whose 1-neighbors $f_1,f_2$ are 1-triangles receives $\frac{1}{18}$ charge from each 1-neighbor of $f_1,f_2$ that is not $f$.
    \item \siii: Each 1-triangle $f$ receives $\frac13 - \ch_2(f)$ charge from its wedge-neighbor.
    \item \siv: Each face $f$ distributes its excess equally over all vertex-neighbors that are 0-pentagons.
    \item \sv: Repeat \siv.
\end{itemize}
As in the last two steps only excesses are contributed, it suffices to show $\ch_3(f) \ge \frac13 v(f)$. As a preparation, we state some facts about the discharging steps. 

\begin{proposition}\label{pro:step2}
    Only 2-triangles contribute in \sii.
\end{proposition}
\begin{proof}
    A face $f$ receiving charge in \sii{} is a 1-triangle with two 1-neighbors $f_1, f_2$ that are 1-triangles, see \cref{fig:fig:discharging-preparation-a}. Then the statement is a consequence of \cref{pro:fill}.
\end{proof}
\begin{figure}[t]
    \hfil
    \begin{subfigure}[b]{0.24\linewidth}
    \center
        \includegraphics[page=2, width = \linewidth]{discharging4planar.pdf}
        \subcaption{}
        \label{fig:fig:discharging-preparation-a}
    \end{subfigure}  
    \hfil
    \begin{subfigure}[b]{0.24\linewidth}
    \center
        \includegraphics[page=1, width = \linewidth]{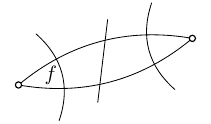}
        \subcaption{}
        \label{fig:fig:discharging-preparation-b}
    \end{subfigure} 
    \hfil
    \caption{(a) In \sii, $f$ can receive $\frac{1}{18}$ charge from two 2-triangles each. (b) A 1-triangle $f$ with a wedge-neighbor that is another 1-triangle creates homotopic edges.}
    \label{fig:discharging-preparation}
\end{figure}

\begin{proposition}\label{pro:wedge-neighbors}
    Only 1-quadrilaterals and 0-faces $f$ with $\vert f \vert \ge 5$ contribute in \siii.
\end{proposition}

\begin{proof}
    2-triangles and 3-triangles do not have a 0-edge, which is necessary to be a wedge-neighbor, and 0-quadrilaterals are per definition no wedge-neighbors. It suffices to observe that by \cref{pro:small-faces,pro:fill,pro:cr-minimal,pro:cr-minimal-2} the only other possible face is a 1-triangle, which would create homotopic multi-edges, see \cref{fig:fig:discharging-preparation-b}.
\end{proof}

\begin{proposition}\label{pro:1-triangle}
    For all 1-triangles $f$, we have $\ch_2(f) \ge \frac{1}{9}$.
\end{proposition}

\begin{proof}
    The initial charge of a 1-triangle $f$ is $\ch(f) = 0$. 1-triangles do not contribute charge in \si{} and by \cref{pro:step2}, this also holds for \sii. The 1-neighbors of $f$ are 1-triangles, 2-triangles or 1-quadrilaterals. So $f$ receives either at least $\frac{1}{9}$ charge in \si{} or $\frac{1}{18}$ charge from two 2-triangles each in \sii, which establishes the claim.
\end{proof}

\begin{proposition}\label{pro:easy-faces}
    After \siii, all faces except 0-pentagons are satisfied. $0$-faces $f$ with $\vert f \vert \ge 6$ have an excess of at least $\frac{1}{9}\vert f \vert$ after \siii.
\end{proposition}

\begin{proof}
Recall that by \cref{pro:small-faces,pro:fill,pro:cr-minimal,pro:cr-minimal-2} only $0$-faces, 1-triangles, 1-quadrilaterals, 2-triangles and 3-triangles exist in $P(\Gamma)$. By assumption, there are no 0-triangles, and by \cref{pro:wedge-neighbors,pro:1-triangle}, the claim holds for 1-triangles. For the following analysis, we use \cref{pro:step2,pro:wedge-neighbors}:
\begin{itemize}
    \item 2-triangles $f$: $\ch(f)=1$, $\ch_1(f) \ge 1- \frac{2 \cdot1}{9} = \frac79$ and $\ch_3(f) = \ch_2(f) \ge \frac79-\frac{2 \cdot 1}{18} = \frac 6 9 \ge \frac 1 3 \cdot 2$.
    \item 3-triangles $f$: $\ch(f)=2$ and $\ch_3(f) = 2 \ge \frac 1 3 \cdot 3$.
    \item 0-quadrilaterals $f$: $\ch(f)=0$ and $\ch_3(f) = 0 \ge \frac 1 3 \cdot 0$.
    \item 1-quadrilaterals $f$: $\ch(f)=1$ and $\ch_2(f) = \ch_1(f) \ge 1 - \frac{2\cdot 1}{9} = \frac{7}{9}$. As $f$ contributes through at most two edges in \siii, we have $\ch_3(f) \ge \frac79-\frac{2\cdot2}9 =\frac{3}{9} \ge \frac{1}{3} \cdot 1$.
    \item $0$-faces $f$ with $\vert f \vert \ge 6$: $\ch(f) = \vert f \vert -4 = \ch_2(f)$ and $\ch_3(f) \ge \vert f \vert-4- \frac29 \vert f \vert$. Thus, $\ch_3(f) \ge 0 + \frac{1}{9} \vert f \vert$ holds.
\end{itemize} 
\end{proof}
\begin{proposition}\label{pro:0-pentagons}
    After \sv, all 0-pentagons are satisfied.
\end{proposition}
\begin{proof}
    Let $f$ be a 0-pentagon. If at most four wedge-neighbors of $f$ are 1-triangles, then we have $\ch_3(f) \ge \frac19$. So assume that all five wedge-neighbors of $f$ are 1-triangles, thus $\ch_3(f) \ge -\frac19$. By \cref{pro:easy-faces}, $f$ receives enough charge in \siv if a vertex-neighbor is a large 0-face. So assume all vertex-neighbors of $f$ are 1-triangles, 2-triangles, 0-quadrilaterals, 1-quadrilaterals or 0-pentagons.

    We introduce some notation for such 0-pentagons: Let $e_i \in E$ for $0 \le i \le 4$ be the edges that form $f$, let 1-triangle $t_i$ be the wedge-neighbor of $f$ at $e_i$ and denote by $v_i$ its vertex. Denote the vertex-neighbors of $f$ at the crossing of $e_i$ and $e_{i+1 \mod 5}$ by $f_i$.
    \begin{itemize}
        \item \emph{Case 1: There is a 2-triangle that is a vertex-neighbor of $f$.} W.l.o.g. let $f_0$ be a 2-triangle. If $\ch_3(f_0) \ge \frac{7}{9}$, then $f_0$ has an excess of $\frac19$ and $\ch_4(f) \ge 0$. If $\ch_3(f_0) < \frac{7}{9}$, then $f_0$ contributes charge in \sii, implying that w.l.o.g. $f_4$ is a 1-triangle whose 1-neighbors both are 1-triangles, see \cref{fig:fig:discharging-1-a}. Then, by the 4-planarity of $e_3$, $f_1$ is also a 2-triangle. So $\ch_2(t_1) = \frac29$ and $f$ contributes only $\frac{1}{9}$ charge to $t_1$ in \siii{}. Thus, $\ch_3(f) \ge 0$.
    \end{itemize}
    In the following, we can assume that wedges of all $t_i$ contain at most one 0-quadrilateral, as otherwise one vertex--neighbor of $f$ is a 2-triangle and we can refer to Case 1.
\begin{figure}[t]
    \hfil
    \begin{subfigure}[b]{0.29\linewidth}
    \center
        \includegraphics[page=3, width = \linewidth]{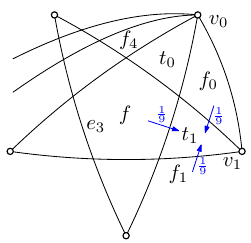}
        \subcaption{}
        \label{fig:fig:discharging-1-a}
    \end{subfigure}  
    \hfil
    \begin{subfigure}[b]{0.29\linewidth}
    \center
        \includegraphics[page=4, width = \linewidth]{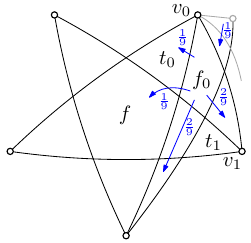}
        \subcaption{}
        \label{fig:fig:discharging-1-b}
    \end{subfigure} 
    \hfil
    \begin{subfigure}[b]{0.29\linewidth}
    \center
        \includegraphics[page=5, width = \linewidth]{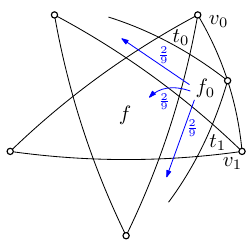}
        \subcaption{}
        \label{fig:fig:discharging-1-c}
    \end{subfigure} 
    \hfil
        \caption{Illustrations for (a) Case 1 and (b)--(c) Case 2.}
        \label{fig:discharging-1}
\end{figure}
    \begin{itemize}
        \item \emph{Case 2: There is a vertex-neighbor of $f$ that is a 1-quadrilateral.} W.l.o.g. let $f_0$ be a 1-quadrilateral.
        \begin{itemize}
            \item \emph{Case 2.1: The vertex of $f_0$ is $v_0$ or $v_1$.} By symmetry, assume that the vertex of $f_0$ is $v_0$. Observe that at most two vertex-neighbors of $f_0$ are possibly 0-pentagons (the third vertex-neighbor is the 1-triangle $t_1$). If $f_0$ does not contribute charge to both its wedge-neighbors, then it has an excess of $\frac{2}{9}$ after \siii{} and so $\ch_4(f) \ge 0$. Otherwise, we have the configuration depicted in \cref{fig:fig:discharging-1-b} and $f$ is the only vertex-neighbor of $f_0$ that is a 0-pentagon. Here, $f_0$ does not contribute charge in \si{} to its 1-neighbor that is not $t_0$ (if it is a 1-triangle, then another 2-triangle pays). Therefore, $\ch_3(f_0) = \frac{4}{9}$ and $f_0$ can contribute its excess of $\frac{1}{9}$ charge to $f$ in \siv.
            \item \emph{Case 2.2: The vertex of $f_0$ is neither $v_0$ nor $v_1$.}  Observe that $f$ is the only vertex-neighbor of $f_0$ that is a 0-pentagon. As a second crossing of the wedges of $t_0,t_1$ would allow a reduction to Case 1, we can assume that the 1-neighbors of $f_0$ are 2-triangles, see \cref{fig:fig:discharging-1-c}. Therefore, $f_0$ does not contribute charge in \si{} and has an excess of at least $\frac{2}{9}$ after \siii. Thus, $f$ is satisfied after \siv.
        \end{itemize}
    
        \item \emph{Case 3: There is a vertex-neighbor of $f$ that is a 0-quadrilateral.} W.l.o.g. let $f_0$ be a 0-quadrilateral.
         \begin{itemize}
            \item \emph{Case 3.1: $f_1$ or $f_4$ is a 0-quadrilateral.} W.l.o.g. let $f_1$ be a 0-quadrilateral. Let $e$ be the edge crossing the wedge of $t_1$, see \cref{fig:fig:discharging-2-a}. As no homotopic multi-edges are allowed, w.l.o.g. the face $f'$ that is a 1-neighbor of $t_0$ and a 0-neighbor of $f_0$ is a 2-face. But $f'$ can not be a 2-triangle, which is a contradiction to \cref{pro:cr-minimal}.         
\begin{figure}[t]
    \hfil
    \begin{subfigure}[b]{0.29\linewidth}
    \center
        \includegraphics[page=6, width = \linewidth]{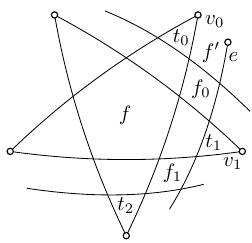}
        \subcaption{}
        \label{fig:fig:discharging-2-a}
    \end{subfigure}  
    \hfil
    \begin{subfigure}[b]{0.29\linewidth}
    \center
        \includegraphics[page=7, width = \linewidth]{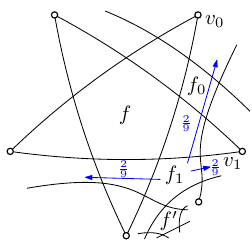}
        \subcaption{}
        \label{fig:fig:discharging-2-b}
    \end{subfigure} 
    \hfil
    \begin{subfigure}[b]{0.29\linewidth}
    \center
        \includegraphics[page=8, width = \linewidth]{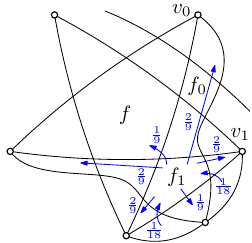}
        \subcaption{}
        \label{fig:fig:discharging-2-c}
    \end{subfigure} 
    \hfil
        \caption{Illustrations for (a) Case 3.1 and (b)--(c) Case 3.2.}
        \label{fig:discharging-2}
\end{figure}
            \item \emph{Case 3.2: $f_1$ or $f_4$ is a 0-pentagon.} W.l.o.g. let $f_1$ be a 0-pentagon. Observe that if $f_1$ has a vertex-neighbor $f'\ne f$ that is a 0-pentagon, then $\ch_3(f_1) \ge \frac{3}{9}$ and $f_1$ can contribute enough charge to $f$ in \siv, see \cref{fig:fig:discharging-2-b}.
            Otherwise, $\ch_3(f_1) = \frac{1}{9}$ implies $\ch_4(f)\ge 0$, and therefore we can assume for the critical case that all wedge-neighbors of $f_1$ are 1-triangles. Here, only the configuration depicted in \cref{fig:fig:discharging-2-c} exists in that the wedge of $t_1$ is not crossed twice. In this configuration, $\ch_3(f_1) = 0$ and $\ch_4(f_1)\ge \frac{1}{9}$, which is enough for $f$ in \sv.
            
            \item \emph{Case 3.3: Both $f_1$ and $f_4$ are 1-triangles.} 
            Then $f_2,f_3$ can not be 0-faces. Also, at most one of them is a 1-triangle, as otherwise there would be an edge homotopic to $e_3$, see \cref{fig:fig:discharging-3-a}. Therefore, $f_2$ or $f_3$ is a larger face (which must be a 1-quadrilateral by \cref{pro:fill,pro:cr-minimal,pro:cr-minimal-2}) and we can refer to Case 2.
        \end{itemize}
\end{itemize}
\begin{figure}[t]
    \hfil
    \begin{subfigure}[b]{0.29\linewidth}
    \center
        \includegraphics[page=9, width = \linewidth]{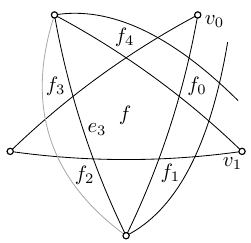}
        \subcaption{}
        \label{fig:fig:discharging-3-a}
    \end{subfigure}  
    \hfil
    \begin{subfigure}[b]{0.29\linewidth}
    \center
        \includegraphics[page=10, width = \linewidth]{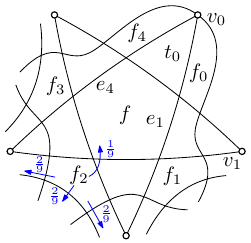}
        \subcaption{}
        \label{fig:fig:discharging-3-b}
    \end{subfigure} 
    \hfil
    \caption{Illustrations for (a) Case 3.3 and (b) Case 4.}
    \label{fig:discharging-3}
\end{figure}
From now on, we can assume that all vertex-neighbors of $f$ are 1-triangles and 0-pentagons.
\begin{itemize}        
        \item \emph{Case 4: There is a vertex-neighbor of $f$ that is a 1-triangle.} W.l.o.g. let $f_0$ be a 1-triangle so that the vertex of $f_0$ is $v_0$.    
        Then $f_4$ is also a 1-triangle. Further, $f_1$ and $f_3$ are 0-pentagons, as otherwise they would be 1-triangles causing edges homotopic to $e_1$ and $e_4$. Finally, this implies that $f_2$ is a 0-pentagon, see \cref{fig:fig:discharging-3-b}. As $f_2$ does not contribute charge to $f_1, f_3$ in \siii, we have $\ch_3(f_2) \ge \frac39$. Here, at most three vertex-neighbors of $f_2$ are 0-pentagons and therefore $f$ can receive at least $\frac{1}{9}$ charge in \siv.
        \item \emph{Case 5: All vertex-neighbors of $f$ are 0-pentagons.}
        For all $f_i$, again we have $\ch_3(f_i) \ge \frac{3}{9}$ and $f$ is satisfied after \siv.
    \end{itemize}
\end{proof}
The combination of \cref{pro:easy-faces} and \cref{pro:0-pentagons} closes the proof of \cref{th:4-planar}.
\end{proof}

\section{Open questions}
We demonstrated how the case analysis in the original proof by Ackerman \cite{DBLP:journals/comgeo/Ackerman19} can be shortened. So one may ask if the complexity of this proof can be reduced further, e.g., by applying the density formula \cite{kaufmann2024density}. Further, our method may open up new possibilities to characterize optimal 4-planar graphs. This approach may lead to improved density bounds also for k-planar graphs for larger $k \ge5$, and also to better constants in the Crossing Lemma.

\bibliography{biblio.bib}

\begin{thebibliography}{10}

\bibitem{ackerman2006maximum}
Eyal Ackerman.
\newblock On the maximum number of edges in topological graphs with no four pairwise crossing edges.
\newblock In {\em Proceedings of the twenty-second annual symposium on Computational geometry}, pages 259--263, 2006.

\bibitem{DBLP:journals/comgeo/Ackerman19}
Eyal Ackerman.
\newblock On topological graphs with at most four crossings per edge.
\newblock {\em Comput. Geom.}, 85, 2019.
\newblock \href {https://doi.org/10.1016/J.COMGEO.2019.101574} {\path{doi:10.1016/J.COMGEO.2019.101574}}.

\bibitem{DBLP:journals/jct/AckermanT07}
Eyal Ackerman and G{\'{a}}bor Tardos.
\newblock On the maximum number of edges in quasi-planar graphs.
\newblock {\em J. Comb. Theory, Ser. {A}}, 114(3):563--571, 2007.
\newblock \href {https://doi.org/10.1016/J.JCTA.2006.08.002} {\path{doi:10.1016/J.JCTA.2006.08.002}}.

\bibitem{DBLP:books/daglib/0019107}
Martin Aigner and G{\"{u}}nter~M. Ziegler.
\newblock {\em Proofs from {THE} {BOOK} {(3.} ed.)}.
\newblock Springer, 2004.

\bibitem{DBLP:conf/compgeom/Bekos0R17}
Michael~A. Bekos, Michael Kaufmann, and Chrysanthi~N. Raftopoulou.
\newblock On optimal 2- and 3-planar graphs.
\newblock In Boris Aronov and Matthew~J. Katz, editors, {\em 33rd International Symposium on Computational Geometry, SoCG 2017, July 4-7, 2017, Brisbane, Australia}, volume~77 of {\em LIPIcs}, pages 16:1--16:16. Schloss Dagstuhl - Leibniz-Zentrum f{\"{u}}r Informatik, 2017.
\newblock \href {https://doi.org/10.4230/LIPICS.SOCG.2017.16} {\path{doi:10.4230/LIPICS.SOCG.2017.16}}.

\bibitem{bungener2025first}
Aaron B{\"u}ngener, Jakob Franz, Michael Kaufmann, and Maximilian Pfister.
\newblock A first view on the density of 5-planar graphs.
\newblock {\em arXiv preprint arXiv:2505.24364}, 2025.

\bibitem{bungener2024improving}
Aaron B{\"u}ngener and Michael Kaufmann.
\newblock Improving the crossing lemma by characterizing dense 2-planar and 3-planar graphs.
\newblock In {\em 32nd International Symposium on Graph Drawing and Network Visualization (GD 2024)}, pages 29--1. Schloss Dagstuhl--Leibniz-Zentrum f{\"u}r Informatik, 2024.

\bibitem{kaufmann2024density}
Michael Kaufmann, Boris Klemz, Kristin Knorr, Meghana M~Reddy, Felix Schr{\"o}der, and Torsten Ueckerdt.
\newblock The density formula: One lemma to bound them all.
\newblock In {\em 32nd International Symposium on Graph Drawing and Network Visualization (GD 2024)}, pages 7--1. Schloss Dagstuhl--Leibniz-Zentrum f{\"u}r Informatik, 2024.

\bibitem{DBLP:journals/dcg/PachRTT06}
J{\'{a}}nos Pach, Rados Radoicic, G{\'{a}}bor Tardos, and G{\'{e}}za T{\'{o}}th.
\newblock Improving the crossing lemma by finding more crossings in sparse graphs.
\newblock {\em Discret. Comput. Geom.}, 36(4):527--552, 2006.
\newblock \href {https://doi.org/10.1007/S00454-006-1264-9} {\path{doi:10.1007/S00454-006-1264-9}}.

\bibitem{DBLP:journals/combinatorica/PachT97}
J{\'{a}}nos Pach and G{\'{e}}za T{\'{o}}th.
\newblock Graphs drawn with few crossings per edge.
\newblock {\em Comb.}, 17(3):427--439, 1997.
\newblock \href {https://doi.org/10.1007/BF01215922} {\path{doi:10.1007/BF01215922}}.

\bibitem{von1983bemerkungen}
R~Von~Bodendiek, Heinz Schumacher, and Klaus Wagner.
\newblock Bemerkungen zu einem sechsfarbenproblem von g. ringel.
\newblock In {\em Abhandlungen aus dem Mathematischen Seminar der Universit{\"a}t Hamburg}, volume~53, pages 41--52. Springer, 1983.

\end{thebibliography}

\end{document}